\documentclass[11pt,a4paper]{article}
\pdfoutput=1 

\usepackage{epsf,epsfig,amsfonts,amsgen,amsmath,amstext,amsbsy,amsopn,amsthm,cases,listings,color
}
\usepackage{ebezier,eepic}
\usepackage{color}
\usepackage{multirow}
\usepackage{epstopdf}
\usepackage{graphicx}
\usepackage{pgf,tikz}
\usepackage{mathrsfs}
\usepackage[marginal]{footmisc}
\usepackage{enumitem}
\usepackage[titletoc]{appendix}
\usepackage{booktabs}
\usepackage{url}
\usepackage{relsize}
\usepackage{mathtools}
\usepackage{comment}
\usepackage{pgfplots}
\usepackage{amssymb}
\usepackage{float}
\usepackage{wasysym}

\usepackage{empheq}

\usepackage{dsfont}

\usepackage{tikz}
\usepackage{longtable}
\usepackage{subfigure}
\pgfplotsset{compat=1.17}
\usepackage{mathrsfs}
\usepackage{wasysym}
\usetikzlibrary{arrows}
\usepackage{aligned-overset}
\usepackage{bm}
\usepackage{bbm}
\usepackage[T1]{fontenc}
\usepackage[nobysame,initials,bibtex-style]{amsrefs}

\allowdisplaybreaks[1]

\newtheorem{definition}{Definition} [section]
\newtheorem{theorem}[definition]{Theorem}
\newtheorem{lemma}[definition]{Lemma}
\newtheorem{proposition}[definition]{Proposition}

\newtheorem{conjecture}[definition]{Conjecture}
\newtheorem{claim}[definition]{Claim}

\theoremstyle{remark}

\newenvironment{poc}{\begin{proof}[Proof of the claim]}{\end{proof}}

\newcommand{\C}[1]{\mathcal{#1}}

\def\ex{\mathrm{ex}}
\def\coex{\mathrm{co\mbox{-}ex}}

\newcommand{\hide}[1]{}

\begin{document}
\title{\bf\Large Rational codegree Tur\'an density of hypergraphs}
\date{\today}
\author{Jun Gao\thanks{Mathematics Institute and DIMAP, University of Warwick, Coventry, UK. Research supported by ERC Advanced Grant 101020255. Email: \texttt{\{jun.gao,o.pikhurko,shumin.sun\}@warwick.ac.uk}}
\and Oleg Pikhurko\footnotemark[1] 
\and Mingyuan Rong\thanks{School of Mathematical Sciences, USTC, Hefei, China. Supported by National Key R\&D Program of China 2023YFA1010201, NSFC Grant No. 12125106 and the USTC Excellent PhD Students Overseas Study Program. Email:\texttt{rong\_ming\_yuan@mail.ustc.edu.cn}}
\and Shumin Sun\footnotemark[1]}

\maketitle
\begin{abstract}
Let $H$ be a \emph{$k$-graph} (i.e.\ a $k$-uniform hypergraph). Its \emph{minimum
codegree} $\delta_{k-1}(H)$ is the largest integer $t$ such that every
$(k-1)$-subset of $V(H)$ is contained in at least $t$ edges of~$H$.
The \emph{codegree
Turán density} $\gamma(\mathcal{F})$ of a family $\mathcal{F}$ of $k$-graphs is the infimum of $\gamma > 0$ such
that every $k$-graph $H$ on $n\to\infty$ vertices with $\delta_{k-1}(H) \ge (\gamma+o(1))\, n$
contains some member of $\mathcal{F}$ as a subgraph.

We prove that, for every integer $k\ge3$ and every rational number $\alpha \in [0,1)$,
there exists a finite family of $k$-graphs $\mathcal{F}$ such that $\gamma(\mathcal{F})=\alpha$. 

Also, for every $k \ge 3$, we establish a strong version of non-principality, namely that there are two $k$-graphs $F_1$ and $F_2$ such that the codegree Tur\'an density of $\{F_1,F_2\}$ is strictly smaller than that of each $F_i$.
This answers a question of Mubayi and Zhao [\emph{J Comb Theory (A)} 114 (2007) 1118--1132].

\end{abstract}

\section{Introduction}
For an integer $k\geq 2$, a \textit{$k$-uniform hypergraph} (in short, \textit{$k$-graph}) $H$ consists of a vertex set $V(H)$ and an edge set $E(H)\subseteq \binom{V(H)}{k}$, that is, every edge is a $k$-element subset of $V(H)$.  Given a (possibly infinite) family $\C F$ of $k$-graphs, the \emph{Tur\'an number} of $\C F$, denoted by $\ex(n,\C F)$, is defined as the maximum number of edges in an $n$-vertex $k$-graph containing no element of $\C F$ as a subgraph. The Tur\'an density $\pi(\C F)$ is defined as
\[
\lim_{n\to \infty}\frac{\ex(n,\C F)}{\binom{n}{k}},
\]
a limit that exists by a standard averaging argument. Tur\'an problems for hypergraphs are notoriously difficult, and our knowledge is quite limited even for seemingly simple cases, for example when $\C F$ forbids the complete $3$-graph on $4$ vertices. For more background, we refer the reader to the surveys~\cite{kee11,sido95}.

In this paper, we investigate the variant called \emph{codegree Tur\'an density}, introduced by Mubayi and Zhao~\cite{MZ07}. 
Given a $k$-graph $H$ and a $(k-1)$-vertex subset $S\subseteq V(H)$, let $d_H(S)$ (or $d(S)$) denote the number of edges of $H$ containing $S$ as a subset. Let $N_H(S)$ (or $N(S)$) denote the \emph{neighbourhood} of $S$, i.e., $N_H(S)\coloneqq\{ v \in V(H)\ : S\cup \{v\}\in E(H)\}$.
The \textit{minimum codegree} $\delta_{k-1}(H)$ of $H$ is defined as the minimum of $d(S)$ over all $(k-1)$-vertex subsets $S$ of $V(H)$. Given a family $\C F$ of $k$-graphs, the \emph{codegree Tur\'an number} $\coex(n,\C F)$ is the maximum value of $\delta_{k-1}(H)$ that an $n$-vertex $\C F$-free $k$-graph $H$ can admit, and the \emph{codegree Tur\'an density} of $\C F$ is defined as
\[
\gamma(\C F)\coloneqq\lim_{n\to \infty}\frac{\coex(n,\C F)}{n},
\]
where the limit is known to exist for any $k$-graph family $\C F$, see~\cite[Proposition 1.2]{MZ07}. 
When $\C F=\{F\}$ is a single $k$-graph, we often omit curly brackets and write $\ex(n,F)$, $\pi(F)$, $\coex(n,F)$ and $\gamma(F)$ instead. 



For each $k\ge 2$, let
\[
\Pi_k\coloneqq \left\{\pi(\C F) \colon \text{$\C F$ is a family of $k$-graphs}\right\}\subseteq [0,1).
\]
The celebrated Erd\H{o}s--Simonovits--Stone Theorem~\cite{ES66,ES46} implies that $\pi(\C F)=\min_{F\in \C F}1-\frac{1}{\chi(F)-1}$ for any graph family $\C F$, where $\chi(F)$ denotes the chromatic number of $F$. This means $\Pi_2=\{0,\frac{1}{2},\frac{2}{3},\dots,
\frac{r-1}{r},\dots\}$. The well-ordered property of $\Pi_2$ motivates the following definition: we call a real number $a\in [0,1)$ a \emph{jump} for $k$ if there exists $\delta>0$ such that no family $\C F$ of $k$-graphs satisfies $\pi (\C F)\in (a,a+\delta)$. Then every real number in $[0,1)$ is a jump for $k=2$. For general uniformity $k$, the study of $\Pi_k$ has a long and rich history, tracing back to the famous Erd\H{o}s jumping conjecture~(see e.g.~\cite{FR84,FPRT07,BT11,P14,CS25}).

Mubayi and Zhao~\cite{MZ07} initiated the study of the analogous problem for codegree Tur\'an density. Similarly, for $k\ge 2$, let
\[
\Gamma_k \coloneqq \left\{\gamma(\C F) \colon \text{$\C F$ is a family of $k$-graphs}\right\}\subseteq [0,1).
\]
For the case $k=2$, it follows from the degree version of Erd\H{o}s--Simonovits--Stone Theorem that $\Gamma_2 = \{0,\frac{1}{2},\frac{2}{3},\dots,\frac{r-1}{r},\dots \}$. In their paper, Mubayi and Zhao~\cite[Theorem~1.6]{MZ07} proved that for the case $k \ge 3$, 
$\Gamma_k$ contains no jump; that is, $\Gamma_k$ is dense in $[0,1)$. They further conjectured that $\Gamma_k=[0,1)$, i.e. for every $0\le \alpha<1$, there exists a family $\C F$ of $k$-graphs such that $\gamma(\C F)=\alpha$. 
Note that in their conjecture, the family is allowed to be infinite. 
In fact, if we restrict attention to finite families only, the conjecture 
becomes trivially false, since the collection of finite families of $k$-graphs  (up to isomorphisms)
is countable, whereas the set of real numbers in the interval $[0,1)$ is uncountable.
Our first result verifies their conjecture for each rational number $\alpha\in [0,1)$ by providing a finite family of $k$-graphs; its proof is inspired by the proof of Theorem~1.6 in~\cite{MZ07}.

\begin{theorem}\label{thm:rational}
    Fix $k\ge 3$. For every rational number $\alpha\in [0,1)$, there exists a finite family $\C F$ of $k$-graphs such that $\gamma(\C F)=\alpha$.
\end{theorem}
It is natural to conjecture that every rational value in $[0,1)$ can be realized 
as the codegree Tur\'an density of a single $k$-graph.

\begin{conjecture}
Fix $k \ge 3$. For every rational number $\alpha \in [0,1)$, there exists a $k$-graph $F$ such that $\gamma(F)=\alpha$.
\end{conjecture}

At present, only a few sporadic examples of $3$-graphs with non-zero codegree Tur\'an density are known. 
In particular,
$\gamma(K_4^{(3)-})=\frac{1}{4}$~\cite{FPVV23}, where $K_4^{(3)-}$ denotes the complete $3$-graph on four vertices 
with one edge removed;
$\gamma(F_{3,2})=\frac{1}{2}$~\cite{Mubayi05}, where $F_{3,2}$ is the Fano plane; 
and $\gamma(C_3^{\ell})=\frac{1}{3}$ for all $\ell \ge 10$ with $3 \nmid \ell$~\cite{Ma2025,PSNS26}, 
where $C_3^{\ell}$ is the tight $3$-uniform cycle of length~$\ell$. Also, Keevash and Zhao~\cite{KZ07} constructed, for every pair of integers $k\ge 3$ and $r \ge 2$, 
a $k$-graph whose codegree Tur\'an density is equal to $\frac{r-1}{r}$.

We call a family $\C F$ of $k$-graphs \emph{non-principal} if its Tur\'an density is strictly less than the density of each member. When $k=2$, Erd\H{o}s--Simonovits--Stone Theorem implies that no such family exists. Motivated by exploring the difference between graphs and hypergraphs, Mubayi and R\"odl~\cite{MR02} conjectured that non-principal families exist for $k\ge 3$. Balogh~\cite{B02} confirmed this conjecture by constructing a finite non-principal family of $3$-graphs. Later, Mubayi and Pikhurko~\cite{MP08} extended this result by constructing a non-principal $k$-graph family of size two for every $k\ge 3$. Mubayi and Zhao~\cite{MZ07} proved a similar result for $\gamma$, showing that there exists a finite family $\C F$ of $k$-graphs such that $0<\gamma(\C F )<\min_{F\in \C F} \gamma(F)$.  Parallel to the situation for $\pi$, they~\cite[Page 1131]{MZ07} posed the problem of constructing two $k$-graphs $F_1$, $F_2$ such
 that $0<\gamma(\{F_1,F_2\})<\min\{\gamma(F_1),\gamma(F_2)\}$. Sudakov (see~\cite[Page 1131]{MZ07}) observed that, for even $k\ge 4$, such a construction can be derived using the framework of~\cite{MP08}. Our next result settles this question for all $k\ge 3$.

 \begin{theorem}\label{thm:non-principal}
  For every $k\ge 3$, there exist $k$-graphs $F_1$ and $F_2$ such that
     \[0<\gamma(\{F_1, F_2\})<\min\{\gamma(F_1),\gamma(F_2)\}.
     \]
 \end{theorem}

\medskip \noindent\textbf{Organisation.} The remainder of this paper is organized as follows.
Theorem~\ref{thm:rational} is proved in Section~\ref{sec:rational}. 
Theorem~\ref{thm:F_r} of Section~\ref{sec:single k-graph} reproves the result of Keevash and Zhao~\cite[Theorem~1.3]{KZ07} that, for every $r$, there is a $k$-graph $F_r^k$ 
whose codegree Tur\'an density is equal to $\frac{r-1}{r}$. While our $k$-graph $F_r^k$ is more complicated than the one in~\cite{KZ07}, its definition is adapted so that $F_2^k$ works with our proof of Theorem~\ref{thm:non-principal} on the existence of a non-principal 2-element $k$-graph family.
Theorem~\ref{thm:non-principal} itself is proved in Section~\ref{sec:non-principal}.

Throughout this paper, $k\ge 3$ is fixed to denote the arity of the considered hypergraphs.

 \section{Proof of Theorem~\ref{thm:rational}}\label{sec:rational}
 
To prove Theorem~\ref{thm:rational}, we need to construct a finite family of $k$-graphs $\mathcal{F}_{\alpha}$ for each rational number $\alpha \in [0,1)$ such that $\gamma(\mathcal{F}_{\alpha}) =\alpha$. The following definition will be used to construct our family of  $k$-graphs.

\begin{definition}[$(s,t)$-extension]
Let $s$ and $t$ be integers with $s \ge t \ge 1$.
For a $k$-graph $F$, let $F(s,t)$ be the family of $k$-graphs constructed as follows. Let $\mathcal{A}(F)= \binom{V(F)}{k-1}$ denote the family of all $(k-1)$-subsets of $V(F)$, and let $\mathcal{T}(F) = \binom{\mathcal{A}(F)}{s}$ be the collection of all $s$-subsets of $\mathcal{A}(F)$. Enumerate $\mathcal{T}(F) = \{T_1,T_2,\dots,T_r\}$ with $r = \binom{|\mathcal{A}(F)|}{s} = \binom{\binom{|V(F)|}{k-1}}{s}$.
Furthermore, define \[
\mathcal{P}(F)\coloneqq\{(P_1,P_2,\dots, P_r): P_i \subseteq T_i \text{ and } |P_i|=t \text{ for each } i\in [r]\}.
\] 
Thus $|\mathcal{P}(F)|=\binom{s}{t}^r$. For each $(P_1,P_2,\dots, P_r) \in \mathcal{P}(F)$, let $F(P_1,P_2,\dots, P_r)$ be the $k$-graph obtained from $F$ by adding $r$ distinct new vertices $v_1, v_2, \dots, v_r$ and, for each $i\in[r]$, adding all edges of the form $\{v_i\} \cup Y$ where $Y\in P_i$. We call the $k$-graph family
\[
F(s,t)\coloneqq \big\{F(P_1,P_2,\dots, P_r):(P_1,P_2,\dots, P_r) \in \mathcal{P}(F) \big\}
\]
 the \emph{$(s,t)$-extension} of $F$. Furthermore, for a family of $k$-graphs, let 
\[\mathcal{F}(s,t) \coloneqq \bigcup_{F\in \mathcal{F}} F(s,t)\] denote the \emph{$(s,t)$-extension} of $\mathcal{F}$. When no confusion arises, we write $\mathcal{A},\mathcal{T}$ and $\mathcal{P}$ for $\mathcal{A}(F),\mathcal{T}(F) $ and $\mathcal{P}(F)$, respectively.

\end{definition}

Roughly speaking, each $k$-graph in $F(s,t)$ is obtained from $F$ by adding vertices corresponding to collections of $s$ distinct $(k-1)$-sets of vertices, where each added vertex forms exactly $t$ edges with $t$ of those $(k-1)$-sets. The family $F(s,t)$  is the collection of all $k$-graphs that can be constructed in this manner.

Here is an example for the 3-graph with a single edge 
$F=\{a,b,c\}$ and $s=2$, $t=1$. 
Then we have \[\mathcal{A} = \{\{a,b\},\{b,c\},\{a,c\}\}\text{ and } \mathcal{T} =\{\{\{a,b\},\{b,c\}\},\{\{a,b\},\{a,c\}\},\{\{b,c\},\{a,c\}\}\}.
\] Then the number of $(P_1,P_2,P_3) \in \mathcal{P}$ is $\binom{s}{t}^r = 2^3$. For each $(P_1,P_2,P_3) \in \mathcal{P}$, we construct a $3$-graph $F(P_1,P_2,P_3)$. For example, for $(P_1,P_2,P_3) = (\{\{a,b\}\},\{\{a,b\}\},\{\{b,c\}\}) \in \mathcal{P}$, $F(P_1,P_2,P_3)$ is the $3$-graph on vertex set \[\{a,b,c,v_1,v_2,v_3\} \text{ and edge set } \{\{a,b,c\},\{a,b,v_1\},\{a,b,v_2\},\{b,c,v_3\}\}.\]
The family $F(s,t)$ is the collection of all $3$-graphs $F(P_1,P_2,P_3)$ with all $(P_1,P_2,P_3)\in \mathcal{P}$. 
We have $|F(s,t)| \le 2^3$, as $|\mathcal{P}| = 2^3$ and some of the $3$-graphs obtained from the above process are isomorphic. In this case, $F(s,t)$ consists of all $3$-graphs obtained by adding 3 new vertices to a single edge $\{a,b,c\}$ with the link of each added vertex being an arbitrary single pair in $\{a,b,c\}$ except we do not allow the same pair to be used 3 times.  

It is clear that $\mathcal{F}(s,t)$ is a finite family whenever $\mathcal{F}$ is finite.

For $ \delta > 0$, let  $M(\delta)$ be the smallest integer such that every $m \ge M(\delta)$ satisfies
\[
m \ge \frac{2(k-1)}{\delta}
\qquad\text{and}\qquad
\binom{m}{k-1} e^{-\delta^{2}(m-k+1)/12} \le \frac{1}{2}.
\]
To construct our family, we also require the following lemma of Mubayi and Zhao in \cite{MZ07}.
\begin{lemma}[\cite{MZ07}*{Lemma~2.1}]
\label{lem:sub-large-co}
Take any $\delta>0$ and $\alpha > 0$ with $\alpha + \delta < 1$.
If $n \ge m \ge M(\delta)$ and $G$ is a $k$-graph on $[n]$ with $\delta_{k-1}(G) \ge (\alpha + \delta)n$, then the number of $m$-sets $S$ satisfying $\delta_{k-1}(G[S]) > \alpha m$ is at least $\frac{1}{2} \binom{n}{m}$. In particular, $G$ contains a subgraph $H$ on $m$ vertices with $\delta_{k-1} (H) > \alpha m$.
\end{lemma}

Now we are ready to prove Theorem~\ref{thm:rational} that every rational number is the codegree density of a finite forbidden family.

\begin{proof}[Proof of Theorem~\ref{thm:rational}.] Take any rational number $\alpha \in [0,1)$. If $\alpha=0$ then we can forbid a single edge, so assume that $\alpha>0$.
Write $\alpha = b/a$ for positive integers $a$ and $b$ with $\gcd(a,b) = 1$. 
Define \[\delta\coloneqq \frac{1}{4a^2} \text{ and } m \coloneqq \max \{4a^2(k-1), M(\delta)\}.\]

Let $\mathcal{H}$ be the collection of all
$k$-graphs $H$ on $m$ vertices with $\delta_{k-1}(H) \ge (\alpha-\delta)m$. We claim that the $(a,b+1)$-extension of $\mathcal{H}$, namely $\mathcal{H}(a,b+1)$, has the codegree Tur\'an density $\alpha$, i.e. $\gamma\left(\mathcal{H}(a,b+1)\right) ={b}/{a} = \alpha$. 

\begin{claim}\label{cl: upper}
    $\gamma(\mathcal{H}(a,b+1))  \ge \frac{b}{a}$.
\end{claim}
\begin{poc}
Given an integer $n$, let $G_{a,b,n}$ be the $k$-graph on the vertex set $[n]$ defined as follows:
\begin{itemize}

    \item Partition $[n]$ into $a$ parts $V_1, V_2, \dots, V_a$ as evenly as possible, that is,
    \[
        \bigl|\,|V_i| - |V_j|\,\bigr| \le 1 \quad \text{for all } i,j \in [a].
    \]

    \item For each $i \in [a]$, add every edge consisting of $k-1$ vertices from $V_i$ and 
    one vertex from $V_j$ whenever
    \[
        j \in \{\, i+1, i+2, \dots, i+b \,\} \pmod{a}.
    \]

    \item Add all edges that intersect at least three of the parts 
    $V_1, V_2, \dots, V_a$.

    \item Add all edges that lie inside some two of the parts $V_1, V_2, \dots, V_a$ and intersect each of these two parts in at least two vertices.
\end{itemize}

It is easy to see that the neighbourhood of any $(k-1)$-set $S$ consists of at least $b$ whole parts, minus the set $S$.
\hide{The neighbourhood of a $(k-1)$-set $S$ behaves as follows.
\begin{itemize}
    \item If $S$ lies entirely within a single part, then its neighbourhood consists of $b$ parts, each of size at least $n/a - 1$.
    \item If $S$ is contained in at least three of the parts 
    $V_1, V_2, \dots, V_a$, then its neighbourhood is the entire set $[n] \setminus S$.
    \item If $S$ is contained in exactly two of the parts $V_1, V_2, \dots, V_a$, then its neighbourhood consists of vertices from at least $a-2 \ge b$ parts, each of size at least $n/a - 1$, when $b \neq a-1$,
and is the entire set $[n]\setminus S$ when $b = a-1$.
\end{itemize}
}%
 Since each part has size at least $n/a-1$, we have
\[
\delta_{k-1}(G_{a,b,n}) \ge \frac{b}{a}\, n - b-|S|.
\]

To show that $\gamma(\mathcal{H}(a,b+1)) \ge \frac{b}{a}$, it suffices to prove that 
$G_{a,b,n}$ is $\mathcal{H}(a,b+1)$-free. 
Suppose, to the contrary, that $G_{a,b,n}$ contains some $k$-graph from 
$\mathcal{H}(a,b+1)$. By the definition of $\mathcal{H}(a,b+1)$, there exists a $k$-graph 
$F \in \mathcal{H}$ such that $G_{a,b,n}$ contains some $k$-graph in $F(a,b+1)$.

With the notation introduced in Definition~2.1, let $\mathcal{A}, \mathcal{T}$, and 
$\mathcal{P}$ be defined with respect to $F$, and set $r \coloneqq |\mathcal{T}|$. 
Then there exists some $P = (P_1,\ldots,P_r) \in \mathcal{P}$ such that 
$G_{a,b,n}$ contains $F(P)$ as a subgraph. In particular, $G_{a,b,n}$ 
contains $F$ as a subgraph.

By the definition of $\mathcal{H}$, we have 
$\delta_{k-1}(F) \ge (\alpha - \delta)m$.
Let  $x_i = |V_i\cap V(F)|$ for $i\in[a]$. 
For convenience, we do index arithmetic modulo $a$, that is, for any $i,j \in \mathbb{Z}$ with  $j \equiv i \pmod a$ we have $x_i = x_j$ and $V_i = V_j$. 

First, let us show  that 
\begin{equation}\label{eq:BSum}
\sum_{j=0}^{b-1}x_{i+j} \ge (\alpha- 2\delta)m,\quad
\mbox{for all $i \in \mathbb{Z}$}.
\end{equation}
Suppose to the contrary that there exists some $i \in \mathbb{Z}$ such that \[\sum_{j=0}^{b-1}x_{i+j} < (\alpha- 2\delta)m.\]
Let $t$ be the minimal integer in $[a]$ such that $x_{i-t} \ge k-1$. If such $t$ exists, then there exists a $(k-1)$-set $S$ of vertices in $V(F) \cap V_{i-t}$. It follows from $\delta_{k-1}(F)\ge (\alpha-\delta)m$ and $N_{F}(S)\subseteq N_{G_{a,b,n}}(S) \subseteq \bigcup_{j=1}^b V_{i-t+j}$ that
\[
(\alpha-\delta)m \le d(S) \le \sum^{b}_{j=1}x_{i-t+j} \le \sum^{t-1}_{j=1}x_{i-j} + \sum_{j=0}^{b-1}x_{i+j} < (k-1)a +(\alpha- 2\delta)m \le (\alpha- \delta)m,
\]
a contradiction. If no such $t$ exists then \[m=|V(F)| = \sum^a_{j=1}x_{i-j}< (k-1)a,\] again a contradiction. This proves~\eqref{eq:BSum}.

Next, we will show that $x_i \ge k-1$ for each $i\in [a]$. By B\'ezout's identity, there exist two positive integers $y,z \in [a]$ such that $by = az+1$. Then, for each $i\in [a]$, we have by~\eqref{eq:BSum} that
\[
y(\alpha-2\delta)m \le \sum^{by-1}_{j=0} x_{i+j} = x_i+ \sum^{az}_{j=1}x_{i+j} = x_i +z\cdot\sum^{a}_{j=1}x_j = x_i+zm.  
\]
This implies that 
\[x_i \ge y(\alpha-2\delta)m -zm= \left(\frac{1}{a}-2\delta  y \right)m \ge \frac{1}{2a}m > k-1.\]

Recall that $G_{a,b,n}$ contains $F(P_1,\dots,P_r)$ as a subgraph for some $(P_1,\ldots,P_r) \in \mathcal{P}$.
Since $F$ intersects each part $V_i$ in at least
$k - 1$ vertices, there exist $a$ distinct vertex sets $S_1,S_2,\dots,S_a$, each of size $k-1$, such that $S_i \subseteq V_i\cap V(F)$ for every $i\in[a]$.
Let $j\in [r]$ be the (unique) index such that $T_j = \{S_1,S_2,\dots,S_a\}$. 
Let $v_j$ be the vertex in $F(P_1, \dots, P_r)$ that forms edges with every element of $P_j$.
Since at most $b$ parts among $V_1,\dots,V_a$ contain a $(k-1)$-set that forms an edge with $v_j$, 
this yields a contradiction to $|P_j|=b+1$.
\end{poc}

\begin{claim}\label{cl: lower}
    $\gamma(\mathcal{H}(a,b+1))  \le \alpha$.
\end{claim}   
\begin{poc}
Let $N$ be the number of vertices in each $k$-graph belonging to the family $\mathcal{H}(a,b+1)$ (which is the same for all $k$-graphs in the family by definition). 
For any $\varepsilon >0$, let $n_0$ be the constant such that 
\[\varepsilon n_0 = \max\{ N, M(\varepsilon) \}.\]
Let $G$ be a $k$-graph on $n \ge n_0$ vertices with $\delta_{k-1}(G) \ge (\alpha+\varepsilon)n$.
To prove that $\gamma(\mathcal{H}(a,b+1)) \le \alpha$, it suffices to show that
$G$ 
contains a subgraph from $\mathcal{H}(a,b+1)$.
By Lemma~\ref{lem:sub-large-co}, $G$  contains $F$ as a subgraph for some $F\in \mathcal{H}$. 
With the notation introduced in Definition~2.1, let $\mathcal{A}, \mathcal{T}$, and 
$\mathcal{P}$ be defined with respect to $F$.
For any element $T = \{S_1,S_2,\dots,S_a\}$ of $\mathcal{T}$,
let $R_T$ be the collection of vertices $v$ in $G$ such that $v$ forms an edge with at least $b+1$ elements in $T$, that is,
 \[
 R_T\coloneqq\big\{v\in V(G): |\{i\in [a]: \{v\}\cup S_i\in G\}|\ge b+1\big\}.
 \] 
 
 We first show that $|R_T|\ge N $ for each $T \in \mathcal{T}$. 
We double count the number of pairs $(S,v)$, where $S \in T$ and 
$v \in V(G)\setminus R_T$, such that $S \cup \{v\} \in E(G)$.
Since each vertex in $V(G)\setminus R_T$ forms an edge with at most $b$ 
elements of $T$, the total number of such pairs is at most $bn$.
On the other hand, from the assumption that
\[
\delta_{k-1}(G) \ge \frac{b}{a}\,n + \varepsilon n
\]
and the fact that $\varepsilon n \ge N$, we obtain that the number of such pairs 
is at least
\[
\biggl(\frac{b}{a}\,n + (N - |R_T|)\biggr)a.
\]
Therefore,
\[
\biggl(\frac{b}{a}\,n + (N - |R_T|)\biggr)a \le bn,
\]
which implies that $|R_T| \ge N$.

Let $r\coloneqq|\mathcal{T}|$ and write
\[
\mathcal{T} = \{T_1, T_2, \dots, T_r\}.
\]
Since $N = |V(F)|+ |\mathcal{T}|$, we can greedily choose distinct vertices
$v_1, v_2, \dots, v_r$ from $V(G) \setminus V(F)$ such that $v_i \in R_{T_i}$ for each $i \in [r]$.
This implies that for each $v_i$ there exists a subset
$P_i \subseteq T_i$ with $|P_i| = b+1$ such that $v_i$ forms an edge with
every element of $P_i$.
Consequently, $G$ contains a copy of
\[
F(P_1, P_2, \dots, P_r) \in F(a, b+1)\subseteq\mathcal{H}(a, b+1),
\]
which completes the proof of the claim.
\end{poc}
Combining Claim~\ref{cl: upper} and Claim~\ref{cl: lower}, we complete the proof of Theorem~\ref{thm:rational}.
\end{proof}

\section{Sequence of $k$-graphs with codegree Tur\'an density $\frac{r-1}{r}$.
}\label{sec:single k-graph}

We say that a $k$-graph $F$ is \emph{$r$-colourable}
if $V(F)$ can be partitioned into $r$ sets $A_1,A_2,\dots,A_r$ such that no edge of $F$ lies entirely inside some part $A_i$. We will use the following observation of Keevash and Zhao~\cite{KZ07}

\begin{proposition}\label{ob:KZ} If a $k$-graph $F$ is not $r$-colourable, then $\gamma(F)\ge \frac{r-1}{r}$.
\end{proposition}

\begin{proof} Let $n\to\infty$ and let $H$ be an $n$-vertex $k$-graph with a partition
\[
V(H)=A_1\cup A_2\cup \dots \cup A_r,
\]
where $|A_i|\in\{\lfloor n/r\rfloor,\lceil n/r\rceil\}$ for every $i\in[r]$ 
and the edge set $E(H)$ consists of all $k$-sets that intersect at least two of these parts.
Clearly, $H$ is $F$-free and satisfies $\delta_{k-1}(H) \ge n-\lceil n/r\rceil$,
which implies that $\gamma(F)\ge \frac{r-1}{r}$.\end{proof}

Next, inductively on $r=1,2,\dots$, we define a $k$-graph $F^k_r$ as follows. The $k$-graph $F^k_1$ is just a single edge. Let $r\ge 2$.
Let $S$ be a vertex set of size $r(k-2)+1$. Let $m\coloneqq \binom{r(k-2)+1}{k-1}$, and enumerate all $(k-1)$-subsets of $S$ as $X_1,X_2,\dots ,X_m$. For each $j\in [m]$, let $\tilde{X}_j$ be a copy of $F^{k}_{r-1}$, with these being vertex disjoint from each other and from the set $S$.
Set 
\[V(F^k_r) \coloneqq S\cup \bigcup^m_{j=1}V\left(\tilde{X}_j\right) \text{ and } E(F^k_r) \coloneqq \Big\{X_j \cup \{v\}: j\in[m],\ v\in V(\tilde{X}_j)\Big\}\cup \bigcup^m_{j=1}E\left(\tilde{X}_j\right) .\]

In order to illustrate the above definition (and since $F_2^k$ appears in Theorem~\ref{thm:non-principal}), let us describe the $k$-graph $F^k_2$ separately.
We start with a vertex set $S$ of size $2k-3$, and enumerate its $(k-1)$-subsets as $X_1,\dots ,X_m$. We add $m$ new edges $\tilde{X}_{1},\dots,\tilde{X}_m$ to $E(F_2^k)$, which are pairwise disjoint and also disjoint from $S$. Finally, for each $i\in [m]$, we add  to $E(F_2^k)$ all $k$-sets that contain $X_i$ and a vertex of $\tilde{X}_{i}$. 
\hide{Thus
\[
V(F^k_2)\coloneqq S\cup \bigcup_{i\in [m]}\tilde{X}_{i}\quad\text{and}\quad
E(F^k_2)\coloneqq \left\{X_i\cup \{v\}:i\in[m],v \in \tilde{X}_{i}\right\}\cup \left\{\tilde{X}_{i}:i\in [m]\right\}.
\]
}

We prove the following result about these $k$-graphs.

\begin{theorem}\label{thm:F_r}
    Fix any $k \ge 3$. Then, for every $r\ge1$, the $k$-graph $F^k_r$ is not $r$-colourable and satisfies $\gamma(F^k_r) = \frac{r-1}{r}$.
\end{theorem}
\begin{proof}
We prove the theorem by induction on~$r$.
Although the case $r=1$ is trivial and the induction step handles the case $r=2$ correctly, we present a separate proof for $r=2$ (since $F_2^k$ plays a crucial role in Theorem~\ref{thm:non-principal}).
\begin{claim}\label{cl:gammaF2}
    The $k$-graph $F^k_2$ is not 2-colourable and $\gamma(F^k_2)=1/2.$
\end{claim}
\begin{poc}
    Let us first show that $F^k_2$ is not $2$-colourable. Suppose, for contradiction, that $F^k_2$ is $2$-colourable, and let $V(F^k_2)=A\cup B$ be the partition such that neither $A$ nor $B$ contains an edge. Recall that $V(F^k_2)= S\cup \bigcup_{i\in [m]}V(\tilde{X}_{i})$, where $|S|=2k-3$. By the pigeonhole principle, there exists a set $S'\subseteq S$ of size $k-1$ entirely lying in $A$ or $B$. Without loss of generality, assume that $S'\subseteq A$. Let $i\in [m]$ be such that $S'=X_i$. Consider the corresponding set $\tilde{X}_{i}$. Since $\tilde{X}_{i}$ is actually an edge of $F^k_2$, it must have at least one vertex $v\in \tilde{X}_{i}$ lying in $A$; otherwise, $\tilde{X}_{i}\subseteq B$ contradicting the assumption that $B$ is edge-free. But then $X_i\cup \{v\}$ is an edge lying in $A$, a contradiction. Hence,  $F^k_2$ is not $2$-colourable.

By Proposition~\ref{ob:KZ}, we have that $\gamma(F^k_2)\ge 1/2$. Let us show the upper bound $\gamma(F^k_2)\le 1/2$. Let $\varepsilon>0$ be an arbitrarily small constant, and assume that $n$ is sufficiently large. We claim that any $k$-graph $H$ with minimum codegree $\delta_{k-1}(H)\ge (1/2+\varepsilon)n$ must contain $F^k_2$ as a subgraph.
        
Identify  an arbitrary $(2k-3)$-subset of $V(H)$ with $S$. Its $(k-1)$-subsets are labelled by $X_1,X_2,\dots,X_m$, where $m\coloneqq \binom{2k-3}{k-1}$. We find suitable sets  $\tilde{X}_1,\dots,\tilde{X}_m\subseteq V(H)$ one by one. Suppose that $i\ge 1$ and we have already defined $\tilde{X}_1,\dots,\tilde{X}_{i-1}$. Let 
\[
A:=N(X_i)\setminus\Big(S\cup \bigcup_{j=1}^{k-1} \tilde{X}_j\Big).
\]By the codegree condition of $H$, we have that
    $$
    |A|\ge d(X_i)-|S|-(i-1)k\ge \left(\frac{1}{2}+\varepsilon\right)n-(2k-3)-(i-1)k\ge 
     k-1.
    $$
    Take an arbitrary $(k-1)$-subset $Y$ of $A$. Since 
    \[
    d(Y)\ge \left(\frac12+\varepsilon\right)n> |V(H)|-|A|,
    \] 
    there must exist a vertex $y\in N(Y)\cap A$. We take $Y\cup \{y\}\in E(H)$ for $\tilde{X}_i$.  After $m$ steps, the sets $S, \tilde{X}_{1},\dots,\tilde{X}_{m}$ form a copy of $F^k_2$ in $H$, finishing the proof of the claim.  
\end{poc}

In the induction step for $r\ge 3$, we have to show that  $F^k_{r}$ is not $r$-colourable and $\gamma(F^k_{r}) =\frac{r-1}{r}$. The proof is very similar to that of Claim~\ref{cl:gammaF2}. 
First, we show that $F^k_r$ is not $r$-colourable.  Suppose, for contradiction, that $F^k_r$ is $r$-colourable. Then $V(F^k_r)$ can be partitioned into $r$ parts $A_1,A_2,\dots,A_r$, each of which induces no edge. Recall that \[V(F^k_r) \coloneqq S\cup \bigcup^m_{j=1}V\left(\tilde{X}_j\right) \text{ and } |S| = (k-2)r+1.\]
By the pigeonhole principle, there exists a set $S' \subseteq S$ of size $k-1$ entirely lying in some part~$A_i$.  Without loss of generality, assume that $S'\subseteq A_r$. Let $S'=X_j$. Consider the corresponding copy of $F^k_{r-1}$, namely $\tilde{X}_j$.
As every vertex of $\tilde{X}_j$ forms an edge together with $X_j$, we must have that
\[
\tilde{X}_j \cap A_r = \varnothing,
\]
which contradicts the fact that $\tilde{X}_j$ is not $(r-1)$-colourable.
Therefore, $F^k_r$ is not $r$-colourable.

Next, we show that $\gamma(F^k_r)=\frac{r-1}{r}$. By Observation~\ref{ob:KZ}, it remains to prove that $\gamma(F^k_r)\le \frac{r-1}{r}$. 
Given any (arbitrarily small) constant $\varepsilon>0$, let $n$ be sufficiently large. Let us show that any $k$-graph $H$ with minimum codegree $\delta_{k-1}(H)\ge (\frac{r-1}{r}+\varepsilon)n$ must contain $F^k_r$ as a subgraph.

Take an arbitrary vertex set $S \subseteq V(H)$ of size $r(k-2)+1$, and label its $(k-1)$-subsets by $X_1,X_2,\dots,X_m$. Our goal is to iteratively find, for each $X_i$, a copy of $F^k_{r-1}$ in $N(X_{i})$, denoted by $\tilde{X_i}$, with the additional requirement that $\tilde{X_1},\dots,\tilde{X_m}$ and $S$ are pairwise disjoint. Suppose that we have obtained $\tilde{X}_1,\dots,\tilde{X}_{i-1}$ for some $i\in[m]$.
We now have to find $\tilde{X}_{i}$.
Define
\[
H_i \coloneqq  H\Big[\, N(X_{i}) \setminus \Big(S \cup \bigcup_{j\in[i-1]} \tilde{X}_j\Big) \,\Big].
\]
Then we have
\[
|V(H_i)|
\ge \left(\frac{r-1}{r}+\varepsilon\right)n  - (r(k-2)+1)
- (i-1)\,|V(F_{r-1}^k)|
\ge \left(\frac{r-1}{r}+\frac{\varepsilon}{2}\right)n,
\]
where the last inequality holds since $n$ is  sufficiently large.
Consequently,
\[
\delta_{k-1}(H_i)
\ge \delta_{k-1}(H) - \bigl(n-|V(H_i)|\bigr)
\ge \left(\frac{r-2}{r-1}+\varepsilon'\right)|V(H_i)|,
\]
for some constant $\varepsilon'>0$ (depending on $\varepsilon$ and $r$ only).
By the induction hypothesis, $\gamma(F^k_{r-1})=\frac{r-2}{r-1}$.
Hence, $H_i$ contains a copy of $F^k_{r-1}$, which we denote by $\tilde{X}_i$.
After we have processed all $i\in[m]$, the sets $S, \tilde{X}_1,\dots,\tilde{X}_m$ form a copy of $F^k_r$ in $H$. This finishes the proof of the theorem.
\end{proof}

 \section{Proof of Theorem~\ref{thm:non-principal}}\label{sec:non-principal}
In this section, let us prove Theorem~\ref{thm:non-principal}. We first present two propositions by Mubayi and Zhao~\cite{MZ07}, which will be needed in our proof. 

Let $K^k(t,t)$ denote the $k$-graph with vertex set $V=A\cup B$ such that $A\cap B=\varnothing$ and $|A|=|B|=t$, where the edge set is defined as $\{S\in \binom{V}{k} \colon |S\cap A|=1\text{ or }|S\cap B|=1\}$. 

\begin{proposition}[\cite{MZ07}*{Proposition~4.2}]
\label{prop:MZ1}
    For $k\ge 3$, there exists a positive integer $\ell=\ell(k)$ such that $\gamma(K^k(\ell,\ell))\ge 1/2$.
\end{proposition}

\begin{proposition}[\cite{MZ07}*{Proposition~4.3}]
\label{prop:MZ2}
    Let $k\ge 3$, $\ell\ge k-1$ and 
    \[
    \rho \coloneqq \frac{1}{2}\left(1-\frac{1}{\binom{\ell}{k-1}}+\frac{1}{\binom{\ell}{k-1}2^{1/\ell}}\right)<\frac{1}{2}.
    \]
    For any $\varepsilon>0$, there exists $N$ such that every $2$-colourable $k$-graph $H$ on $n$ vertices with $n>N$ and $\delta_{k-1}(H)\ge (\rho+\varepsilon)n$ contains a copy of $K^k(\ell,\ell)$.
\end{proposition}

Now we are ready to prove Theorem~\ref{thm:non-principal}.

\begin{proof}[Proof of Theorem~\ref{thm:non-principal}]
    Given an integer $k\ge 3$, let $\ell=\ell(k)$ be the positive integer returned by Proposition~\ref{prop:MZ1}. Define $F_1\coloneqq K^k(\ell,\ell)$ and $F_2 \coloneqq F^k_2$, where  $F^k_2$ is the $k$-graph defined in Section~\ref{sec:single k-graph}.
    
    Let us show that $F_1$ and $F_2$ satisfy Theorem~\ref{thm:non-principal}. By Proposition~\ref{prop:MZ1} and Theorem~\ref{thm:F_r}, we  know that $\gamma(F_1)\ge 1/2$ and $\gamma (F_2) =\frac{1}{2}$. Thus it is enough to show that $\gamma(\{F_1,F_2\})<1/2$.

Given $k$ and $\ell$, let $\rho =\rho(k,\ell)<1/2$ be the constant returned by Proposition~\ref{prop:MZ2}. Set 
$$\varepsilon\coloneqq \frac{1}{6}\left(\frac{1}{2}-\rho\right)>0,$$
        and take $n$ to be sufficiently large. It is enough to prove that any $n$-vertex $k$-graph $H$ with minimum codegree $\delta_{k-1}(H)\ge (1/2-\varepsilon)n$ must contain $F_1$ or $F_2$ as a subgraph, for this yields $\gamma(\{F_1,F_2\})\le 1/2-\varepsilon<1/2$ as desired.

Assume that $H$ is $F_2$-free. Let us show that, for any vertex subset $S$ of size $2k-3$, there must exist a $(k-1)$-subset $X\subseteq S$ such that one can find a set $A\subseteq N(X)$ with $|A|\ge (1/2-2\varepsilon)n$ containing no edge of $H$. Suppose on the contrary that no such $X$ exists. We enumerate $(k-1)$-subsets of $S$ as $X_1,\dots,X_m$, $m\coloneqq {2k-3\choose k-1}$, and follow a similar inductive procedure as in the proof of the upper bound in Claim~\ref{cl:gammaF2} to find edges $\tilde{X}_i$ for $i=1,2,\dots,m$. In fact, the argument becomes slightly simpler here, since at Step $i$, we may directly choose any edge in the set $N(X_{i})\setminus (S\cup \bigcup_{j\in [i-1]}\tilde{X}_{j})$. The latter set has at least $\delta_{k-1}(H)-(2k-3)-(i-1)k\ge (1/2-2\varepsilon)n$ vertices and thus spans at least one edge by our assumption. This eventually yields a copy of $F_2$, a contradiction.

Fix an arbitrary vertex subset $S\subseteq H$ of size $2k-3$. Let $X\subseteq S$ and $A\subseteq N(X)$ be two sets satisfying the property above. Then $A$ is an independent set and $|A|\ge (1/2-2\varepsilon)n$. Now take any $(2k-3)$-set $S' \subseteq A$. Then there exists a $(k-1)$-set $X' \subseteq S'$ such that
        there is a vertex set $B \subseteq N(X')$ of size at least \(\left(1/2-2\varepsilon\right)n\) which contains no edge of $H$. Since $A$ is independent, we have $B\cap A\subseteq N(X')\cap A=\varnothing$. Let $H'=H[A\cup B]$ be the induced $k$-graph on $A\cup B$. Then $|V(H')|\ge |A|+|B|\ge (1-4\varepsilon)n$, and thus
        \[
        \delta_{k-1}(H')\ge \delta_{k-1}(H)-4\varepsilon n\ge \left(\frac{1}{2}-5\varepsilon\right)n\ge (\rho+\varepsilon)n\ge (\rho+\varepsilon)\cdot |V(H')|.
        \]
        Since both $A$ and $B$ are independent, $H'$ is $2$-colourable. By Proposition~\ref{prop:MZ2}, $H'$ contains a copy of $F_1$. This finishes the proof of the theorem.
\end{proof}

\bibliographystyle{abbrv}
\bibliography{codegree}
\end{document}